\documentclass[11pt]{article}

\usepackage{enumerate}
\usepackage{latexsym} 
\usepackage{theorem}
\usepackage{graphics}
\usepackage{graphicx}
\usepackage{amsmath}
\usepackage{amssymb}
\usepackage[english]{babel} 
\usepackage{comment}
\usepackage{color}
\usepackage{tikz}
\usetikzlibrary{arrows}
\usepackage{ifpdf}

\newtheorem{defeng}{Definition}[section]
\newtheorem{theorem}[defeng]{Theorem}

{\theorembodyfont{\rmfamily} }
{\theorembodyfont{\rmfamily} }
{\theorembodyfont{\rmfamily} }
{\theoremstyle{break}\theorembodyfont{\rmfamily} }
{\theoremstyle{break}\theorembodyfont{\rmfamily} }

\def\d{\hbox{-}}

\newcommand{\poubelle}[1]{}

\tikzstyle{noeud}=[circle,inner sep=0, minimum size =12 pt, line width = 0.5pt, draw=black, fill=white]

\newcounter{claim}

\newenvironment{proof}[1][]%
 {\noindent {\setcounter{claim}{0}\sc proof ---
   }{#1}{}}{\hfill$\Box$\vspace{2ex}} 

\newenvironment{claim}[1][]%
{\refstepcounter{claim}\vspace{1ex}\noindent{(\it\arabic{claim}){#1}{}}\it}{\vspace{1ex}}

\newenvironment{proofclaim}[1][]%
	{\noindent {}{#1}{}}{ This proves~(\arabic{claim}).\vspace{2ex}}

\newcommand{\sm}{\setminus} 

\usepackage{ifpdf}

\ifpdf
\fi

\title{Wheel-free planar graphs}

\author{Pierre Aboulker\thanks{Partially supported by \emph{Agence Nationale de la Recherche} under reference    \textsc{anr 10 jcjc 0204 01}.}\\
Concordia University, Montr\'eal, Canada\\ email: pierreaboulker@gmail.com
\\
\\
Maria Chudnovsky\thanks{Supported by NSF grants DMS-1001091 and IIS-1117631.}\\
Columbia University, New York, NY 10027, USA \\ e-mail: mchudnov@columbia.edu
\\
\\
Paul Seymour\thanks{Supported by ONR grant N00014-10-1-0680 and NSF grant 
DMS-1265563.}\\
Princeton University, Princeton, NJ 08544, USA \\ e-mail: pds@math.princeton.edu
\\
\\
Nicolas Trotignon\thanks{Partially supported by ANR project Stint
  under reference ANR-13-BS02-0007.  Also Labex Milyon, INRIA, Universit\'e de Lyon, Universit\'e Lyon~1.}\\
CNRS, LIP, ENS de Lyon, Lyon, France\\ e-mail: nicolas.trotignon@ens-lyon.fr}
    
\begin{document}

\maketitle

\begin{abstract}
A \emph{wheel} is a graph formed by a chordless cycle $C$ and a vertex
$u$ not in $C$ that has at least three neighbors in $C$.  We prove
that every 3-connected  planar graph that does not contain a wheel as
an induced subgraph is either a line graph or has a clique cutset. We
prove that every  planar graph that does not contain a wheel as
an induced subgraph is 3-colorable.

\end{abstract}

\section{Introduction}
All graphs in this paper are finite and simple.  A graph $G$
\emph{contains} a graph $F$ if an induced subgraph of $G$ is
isomorphic to $F$.  A graph $G$ is {\em $F$-free} if $G$  does not
contain $F$.  For a set of graphs $\mathcal F$, $G$ is \emph{$\mathcal
  F$-free} if it is $F$-free for every $F \in \mathcal F$.   An
\emph{element} of a graph is a vertex or an edge.  When $S$ is a set
of elements of $G$, we denote by $G \sm S$ the graph obtained from $G$
by deleting all edges of $S$ and all vertices of $S$. 

A \emph{wheel} is a graph formed by a chordless cycle $C$ and a vertex
$u$ not in $C$ that has at least three neighbors in $C$.  Such a wheel
is denoted by $(u,C)$; $u$ is the \emph{center} of the wheel
and $C$ the \emph{rim}.  Observe that $K_4$ is a wheel (in some papers
on the same subject, $K_4$ is not considered as a wheel).  Wheels
  play an important role in the proof of several decomposition
  theorems.   Little is
known about wheel-free graphs.   The only positive result is due to Chudnovsky
(see~\cite{abChTrVu:moplex} for a proof).  It states that every non-null
wheel-free graph contains a vertex whose neighborhood is  made of
disjoint cliques with no edges between them.  No bound is known on the 
chromatic number of
wheel-free graphs.  No decomposition theorem is known for wheel-free
graphs.  However,  several classes of wheel-free
graphs were shown to have a structural description.

\begin{itemize}
\item Say that a graph is \emph{unichord-free} if it does not contain
  a cycle with a unique chord as an induced subgraph.  The class of
  \{$K_4$, unichord\}-free graphs is a subclass of wheel-free graphs
  (because every wheel contains a $K_4$ or a cycle with a unique chord
  as an induced subgraph), and unichord-free graphs have a complete
  structural description, see~\cite{nicolas.kristina:one}.
\item It is easy to see that the class of  graphs that do
  not contain a subdivision of a wheel as an induced subgraph is the
  class of graphs that do not contain a wheel or a subdivision of
  $K_4$ as induced subgraphs.  Here again, this subclass of wheel-free
  graphs has a complete structural description,
  see~\cite{nicolas:isk4}.
\item The class of graphs that do not contain a wheel as a subgraph
  does not have a complete structural description so far.  However, in
  \cite{thomassenToft:k4} (see also \cite{aboulkerHT:wheelFree}),
  several structural properties for this class are given.  
\item A \emph{propeller} is a graph formed by a chordless cycle $C$
  and a vertex $u$ not in $C$ that has at least \emph{two} neighbors
  in $C$.  So, wheels are just special propellers, and the class of
  propeller-free graphs is a subclass of wheel-free graphs.  In
  \cite{aboulkerRTV:propeller}, a structural description of
  propeller-free graphs is given. 
\end{itemize}

Interestingly, every graph that belongs to one of the four classes
described above is 3-colorable (this is shown in cited papers).
One might conjecture that every wheel-free graph is 3-colorable, but
this is false as shown by the graph represented on
Figure~\ref{fig:R} (it is wheel-free and has chromatic number~4).  
Also, the four classes have polynomial time
recognition algorithms, so one could conjecture that so does the class
of wheel-free graphs.  But it is proved in~\cite{diotTaTr:13} that it
is NP-hard to recognize them.  All this suggest that possibly, no
structural description of wheel-free graphs exists.

 \begin{figure}[hbtp]
\begin{center}
\includegraphics{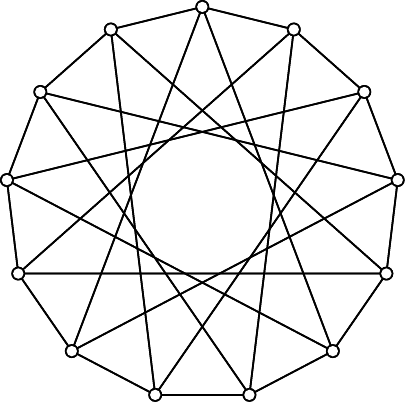}
\end{center}
\caption{The Ramsey graph $R(3, 5)$, the unique graph $G$
  satisfying $|V(G)|\geq 13$, $\alpha(G) = 4$ and $\omega(G) = 2$.\label{fig:R}}
\end{figure}

 \medskip 

 In this paper, we study planar wheel-free graphs.  A \emph{clique
   cutset} of a graph $G$ is a clique $K$ such that $G\sm K$ is
 disconnected.  When the clique has size three, it is referred to as a
 \emph{$K_3$-cutset}.  When $R$ is a graph, the line graph of $R$ is
 the graph denoted by $L(R)$ defined as follows: the vertex-set of
 $L(R)$ is $E(G)$, and two vertices $x$ and $y$ of $L(R)$ are adjacent
 if they are adjacent edges of $R$. We prove the following theorems.

\begin{theorem}
  \label{th:struct}
  If $G$ is a 3-connected wheel-free planar graph, then either $G$ is
  a line graph or $G$ has a clique cutset.
\end{theorem}

We now give a complete description of 3-connected wheel-free planar
graphs, but we first need some terminology.  A graph is \emph{basic} if it
is the line graph of a graph $H$ such that either $H$ is $K_{2,3}$, or
$H$ can be obtained from a 3-connected cubic planar graph by
subdividing every edge exactly once.  We need to name four special graphs: the
claw, the diamond, the butterfly and the paw, that are represented in
Figure~\ref{f:cdbp}.  Basic graphs have a simple characterization
given below.

 \begin{figure}
\begin{center}
\includegraphics[height=1cm]{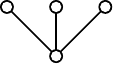}\rule{1em}{0ex}
\includegraphics[height=1cm]{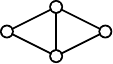}\rule{1em}{0ex}
\includegraphics[height=1cm]{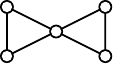}\rule{1em}{0ex}
\includegraphics[height=1cm]{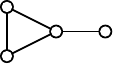}
\end{center}
\caption{The claw, the diamond, the butterfly and the paw.\label{f:cdbp}}
\end{figure}

\begin{theorem}
  \label{th:equiv}
  Let $G$ be a graph. The following statements are equivalent.
  \begin{enumerate}
  \item $G$ is basic.
  \item $G$ is a 3-connected wheel-free planar line graph.
  \item $G$ is 3-connected, planar and \{$K_4$, claw, diamond,
    butterfly\}-free.
  \end{enumerate}
\end{theorem}

With Theorems~\ref{th:struct} and~\ref{th:equiv}, we may easily prove
the complete description of 3-connected wheel-free planar graphs. By
the Jordan curve theorem, a simple closed curve $C$ in the plane
partitions its complement into a bounded open set and an unbounded
open set.  They are respectively the \emph{interior} and the
\emph{exterior} of $C$.

\begin{theorem}
\label{complete}
The class $\cal C$ of 3-connected wheel-free planar graphs is the
class of graphs that can be constructed as follows: start with basic
graphs and repeatedly glue previously constructed graphs along cliques
of size three that are also face boundaries.
\end{theorem}

\begin{proof}
  By Theorem~\ref{th:equiv}, a basic graph is in $\cal C$.  Also
  gluing along cliques of size three that are also face boundaries
  preserves being in $\cal C$ (in particular, it does not create wheels,
  because wheels have no clique cutset). It follows that the
  construction only constructs graphs in $\cal C$.

  Conversely, let $G$ be a graph in $\cal C$.  We prove by induction
  on $|V(G)|$ that $G$ can be constructed as we claim.  If $G$ is a
  line graph, then Theorem~\ref{th:equiv} implies $G$ is basic.  So by
  Theorem~\ref{th:struct} we may assume that G has a clique cutset.
  Since $G$ is 3-connected, this clique must be a triangle $K$ whose
  edges form a closed curve in the plane since $G$ is planar.  So, $G$
  is obtained by gluing along $K$ the two induced subgraphs of $G$
  that are drawn respectively on the closure of the interior and on
  the closure of the exterior of $K$.  These two graphs are easily
  checked to be 3-connected because $G$ is 3-connected.  It follows by
  induction that $G$ can be constructed from previously constructed
  graphs by gluing along a triangle that is also a face boundary.
\end{proof}

A consequence of our description is the following. 

\begin{theorem}
  \label{th:col}
  Every wheel-free planar graph is 3-colorable.
\end{theorem}

We have no conjecture (and no theorem) about the structure of
wheel-free planar graphs in general (possibly not 3-connected).  In
Figure~\ref{f:2c} three wheel-free planar graphs of connectivity~2
are represented.  It can be checked that they belong to none
of the four classes described above (each of them contains a cycle
with a unique chord, an induced subdivision of $K_4$, a wheel as a
subgraph and a propeller).  So we do not understand them.  We leave
the description of the most general wheel-free planar graph as an open
question.

\begin{figure}
 \hfill \includegraphics{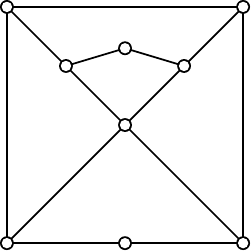}\hfill \includegraphics{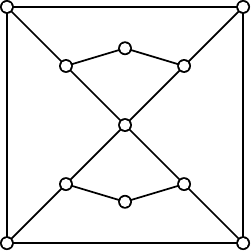}\hfill \includegraphics{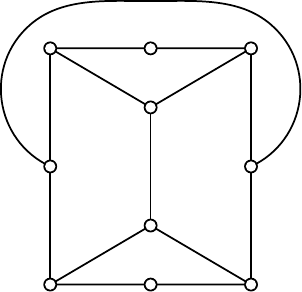}
  \hfill
\caption{Some wheel-free planar graphs\label{f:2c}}
\end{figure}

\medskip

Section~\ref{sec:struct} gives the proof of Theorem~\ref{th:struct},
and in fact of a slight generalization that we need in Section~\ref{coloration}.
 In
Section~\ref{sec:equiv}, we prove Theorem~\ref{th:equiv}. 
Theorem~\ref{th:col} is proved in Section~\ref{coloration}.

\subsection*{Notation, definitions and preliminaries}

We use notation and classical results from~\cite{diestel:graph}.  Let
$G$ be a graph, $X \subseteq V(G)$ and $u\in V(G)$.  We denote by
$G[X]$ the subgraph of $G$ induced on $X$,  by $N(u)$ the
set of neighbors of $u$, and by $N(X)$ the set of vertices of $V(G) \sm X$
adjacent to at least one vertex of $X$; and we define $N_X(u)=N(u) \cap
X$.  We sometimes write $G\sm u$ instead of $G\sm \{u\}$.   When $e$
is an edge of $G$, we denote by $G/e$ the graph obtained from $G$ by
contracting~$e$.

A {\em path} $P$ is a graph with $k\ge 1$ vertices that can be numbered $p_1,\ldots, p_k$, and with $k-1$ edges
$p_ip_{i+1}$ for $1\le i<k$.
The vertices $p_1$ and $p_k$ are the {\em end-vertices} of $P$,
and $\{p_2,\ldots, p_{k-1}\}$ is the {\em interior} of $P$. We also say that $P$ 
is a {\em $p_1p_k$-path}.  If $P,Q$ are paths, disjoint except that they have one end-vertex $v$ in common, then their
union is a path and we often denote it by $P\d v\d Q$. If $a,b$ are vertices of a path $P$, we denote
the subpath of $P$ with end-vertices $a,b$ by $a\d P\d b$.

A {\em cycle} $C$ is a graph with $k\ge 3$ vertices that can be numbered $p_1,\ldots,p_k$, and with $k$ edges
$p_ip_{i+1}$ for $1\le i\le k$ (where $p_{k+1} = p_1$).

Let $Q$ be a path or a cycle in a graph $G$.  The {\em length} of $Q$ is the number
of its edges.  An edge $e=xy$ of $G$ is a {\em chord} of $Q$ if $x,y\in
V(Q)$, but $xy$ is not an edge of $Q$.  A chord is \emph{short} if its
ends are joined by a two-edge path in $Q$.

  We need the following.

\begin{theorem}[Harary and Holzmann~\cite{harary.holzmann:lgbip}]\label{harary} 
  A graph is the line graph of a triangle-free graph if and only if it is
  \{diamond, claw\}-free.
\end{theorem}

\section{Almost 3-connected wheel-free planar graphs}
\label{sec:struct}

A graph $G$ is \emph{almost 3-connected} if either it is 3-connected
or it can be obtained from a 3-connected graph by subdividing one edge
exactly once.  For a 2-connected graph drawn in the plane, the
boundary of every face is a cycle.  We need the following consequence.

\begin{theorem}
  \label{planar}
  Let $G$ be an almost 3-connected graph drawn in the plane, and let
  $x$ be a vertex of $G$ such that all its neighbors have degree at
  least three.  Let $R$ be the face of $G \sm \{x\}$ in which $x$ is
  drawn. Then the boundary of $R$ is a cycle $C$, and $C$ goes through
  every vertex of $N(x)$.
\end{theorem}

In this section, we prove the theorem below, which clearly implies
Theorem~\ref{th:struct}.  We prove the stronger statement below
because we need it in the proof of Theorem~\ref{th:col}.

\begin{theorem} 
  \label{claw} 
  If $G$ is an almost 3-connected wheel-free planar graph with no
  clique cutset, then $G$ is a line graph.
\end{theorem}

\begin{proof}
  The proof is by contradiction, so suppose that $G$ is an almost 3-connected 
  wheel-free planar graph that has no clique cutset and that is not a line
  graph.

\begin{claim} 
  \label{trianglefact}
  Let $\{a,b,c\}$ be a clique of size three in $G$, and let $P$ be a
  chordless path of $G \setminus \{b,c\}$ with one end $a$. Then at
  least one of $b,c$ has no neighbor in $V(P) \setminus \{a\}$.
\end{claim}

\begin{proofclaim}
Suppose $b,c$ both have neighbors in $V(P) \setminus \{a\}$, and $P'$ be the minimal subpath of $P$
such that $a \in V(P')$, and both $b$ and $c$ have neighbors in
$V(P') \setminus \{a\}$. We may assume
that $P'$ is from $a$ to $x$, $x$ is adjacent to $b$, and $b$ has no neighbor
in $V(P') \setminus \{a,x\}$.  Then  $a\d P'\d x\d b\d a$ is an induced cycle, say $C$.
Now since $c$ is adjacent to $a$ and $b$, and has a neighbor in 
$V(P') \setminus \{a\}$, it follows that $(c,C)$ is a wheel, a contradiction.
\end{proofclaim}

\begin{claim} 
  \label{diamond}
  $G$ is diamond-free.
\end{claim}

\begin{proofclaim}
  Suppose that $\{a, x, b, y\}$ induces a diamond of $G$, and
  $xy\notin E(G)$.  Since $\{a,b\}$ is not a cutset of $G$, there
  exists a chordless $xy$-path $P$ in $G \setminus \{a,b\}$, contrary 
  to~(\ref{trianglefact}).
\end{proofclaim}

A vertex $e$ of $G$ is a \textit{corner} if $e$ has degree two, and there 
exist four vertices $a, \, b, \, c, \, d$ such that 
$E(G[\{a,b,c,d,e\}])=\{ab, ac, bc,cd,de,eb\}$.

\begin{claim} \label{corner}
  No vertex of $G$ is a corner.
\end{claim}

\begin{proofclaim}
  Suppose that $e \in V(G)$ is a corner and let $a, \, b, \, c, \, d$
  be four vertices as in the definition.  Since $\{b,c\}$ is not a
  cutset of $G$, there exists a chordless $ad$-path $P$ in $G \sm
  \{b,c\}$.  But now the path  $a\d P\d d\d e$ contradicts (\ref{trianglefact}). 
\end{proofclaim}

\begin{claim}
  \label{c:claw}
  $G$ contains a claw. 
\end{claim}

\begin{proofclaim}
  Otherwise, by~(\ref{diamond}) and Theorem~\ref{harary}, $G$ is a
  line graph, a contradiction.
\end{proofclaim}

\bigskip

The rest of the proof is in two steps.  We first prove the
existence of a special cutset, called an ``I-cutset''
(defined below).  Then we use the I-cutset to obtain a contradiction.

\bigskip

Let $\{u,x,y\}$ be a cutset of size three of $G$.
A component of $G \setminus \{u,x,y\}$ is said to be {\em degenerate} if it has
only one vertex, or it has exactly two vertices $a,b$ and $G[\{u, x, y, a, b\}]$ has the following edge-set: $\{xy, ax, ay,
ab, bu\}$, and {\em nondegenerate} otherwise.

A cutset $\{u,x,y\}$ of size three of $G$ is an {\em $I$-cutset} if $G[\{u,x,y\}]$ has
at least one edge and $G \setminus \{u,x,y\}$ has at least two
connected components that are non-degenerate.

\begin{claim}\label{Icutset}
$G$ admits an $I$-cutset.
\end{claim}

\begin{proofclaim}
Fix a drawing of $G$ in the plane.
  By~(\ref{c:claw}), $G$ contains a claw. Let $u$ be the center of a claw.
Let $u_1', u_2,
  \ldots, u_k$ ($k \ge 3$) be the neighbors of $u$, in cyclic order around $u$,
where $u_2,\ldots, u_k$ have degree at least three. 
If $u_1'$ has degree two, let $u_1$ be its neighbor different from $u$, and
otherwise let $u_1=u_1'$.

Deleting $u$, and also deleting $u_1'$ if $u_1'$ has degree two,
yields a 2-connected graph, drawn in the plane, and therefore, the
face $R$ of this drawing in which $u$ is drawn is bounded by a cycle
$C$. By Theorem~\ref{planar} $u_1,u_2,\ldots, u_k$ all belong to $C$,
and are in order in $C$.  For $i=1, \ldots, k$, let $S_i$ be the
unique $u_iu_{i+1}$-path included in $C$ that contains none of
$u_1,\ldots, u_k$ except $u_i$ and $u_{i+1}$ (subscripts are taken
modulo $k$).

Assume that $xy$ is a chord
  of~$C$.  Vertices $x$ and $y$ edge-wise partition $C$ into two
  $xy$-paths, say $P'$ and $P''$.  Since $R$ is a face of $G \sm
  \{u\}$ or of $G\sm \{u,u_1'\}$, it follows that
$\{u,x,y\}$ is a cutset of $G$ that separates the interior of $ P'$ from
  the interior of $ P''$.  If $xy$ is not a short chord, then both these interiors 
  contain at least two vertices and therefore $\{u,x,y\}$ is an
  $I$-cutset of $G$.  So we may assume that $xy$ is
  short.  If $x,y$ both belong to
  $S_i$ for some $i$, then $\{x,y\}$ is a clique-cutset of $G$, a
  contradiction.  
Thus we may assume
that for every chord $xy$ of $C$, there exists $i \in \{1, \dots, k\}$
such that $x \in S_{i-1}$, $y \in S_i$ and both $xu_i$ and $yu_i$ are edges.

\medskip

\noindent{\bf Claim}.  $k=3$ and  $u'_1$ has degree two. 
\medskip

\noindent To prove the claim, assume by way of contradiction that $u$ has at least three neighbors of degree at least $3$.
Since $G$ is wheel-free, $C$ must have chords.  Let $xy$ be a chord, and choose $i \in
\{1, \dots, k\}$ such that $xu_i$ and $yu_i$ are edges of $C$.
Suppose first that we cannot choose $xy$ and $i$ such that $u_i$ is adjacent to $u$. Consequently
$i=1$, and $u_1'$ has degree two; moreover,
the cycle obtained from $C$ by replacing the edges $xu_1$ and $u_1y$ by $xy$ is induced. 
Since in this case $k\ge 4$, it follows that $u$ has at least three neighbors in this cycle
and so $G$ contains a wheel, a contradiction. 

We can therefore choose $xy$ and $i$ 
such that $u_i$ is adjacent to $u$. 
It follows that $u_{i+1},u_{i-1}$ are not consecutive in $C$, since
$u$ is the center of a claw.  
We claim that there are no edges between
$S_i \sm \{u_i\}$ and $S_{i-1} \sm \{u_i\}$,
except $xy$.  For suppose such an edge exists, say $ab$. 
Since $u_{i+1},u_{i-1}$ are not consecutive in $C$, it follows that $ab$ is
a chord of $C$.  Since every chord of $C$ is short, it follows that
$\{a,b\}=\{u_{i+1},u_{i-1}\}$ and that $u_{i+1}u_{i+2}$ and
$u_{i+2}u_{i-1}$ are both edges of $C$.  But now, $G[{u, u_{i+1},
  u_{i-1}, u_{i+2}}]$ is a wheel or, in case one of $i-1$, $i+1$ or
$i+2$ equals $1$ and $u'_1$ has degree $2$, $G[{u, u_{i+1}, u_{i-1},
  u_{i+2}, u'_1}]$ is a wheel.

Hence there are no edges between $S_i \sm \{u_i\}$ and $S_{i-1} \sm \{u_i\}$ except $xy$ and thus,
$$u \d u_{i-1}\d S_{i-1}\d x \d y \d S_i \d u_{i+1} \d u$$
or, in the case where $i-1=1$ and $u'_1$ exists, 
$$u \d u_1' \d  u_{1}\d S_{1} \d x \d y \d S_1 \d u_{3} \d u$$
is an induced cycle containing three neighbors of $u_i$, a contradiction. 
This proves the claim.

\medskip

Observe that the claim implies that every center of a claw in $G$ has
degree three and is adjacent to $u'_1$ since $G$ has at most one
vertex of degree two.

Let $x,y$ be the neighbors
of $u_2$ in $S_1$, $S_2$ respectively.  Note that possibly
$x=u_1$.  Observe that, since $u$ is the center of a claw, $y \neq
u_3$.  Since every center of a claw is adjacent to $u_1'$, it follows that
$u_2$ is not the center of a claw and thus $xy$ is an edge.  Now,
$$x\d y\d S_2\d u_3\d u\d u_1'\d u_1\d S_1\d x$$
must admit a chord, for otherwise
$u_2$ is the center of a wheel of $G$. 
Hence $u_1u_3$ is an edge.
Let $z$ be the neighbor of $u_3$ in $S_2$.  Since $u_3$ is not
the center of claw, $u_1z$ is an edge and thus $u_1'$ is a corner, a
contradiction to~(\ref{corner}). 
\end{proofclaim}
\bigskip

\begin{claim} 
  \label{path} 
  Let $\{u,x,y\}$ be an $I$-cutset of $G$ where $xy$ is an edge
  and let $C$ be a nondegenerate connected component of $G \sm \{u,x,y\}$.
  Then there exist $v\in \{x,y\}$ and a path $P$ of $G[C \cup \{u,x,y\}]$ from 
  $u$ to $v$, such that the vertex of $\{x,y\} \setminus \{v\}$ has
  no neighbor in $V(P) \setminus \{v\}$. 
\end{claim}

\begin{proofclaim}
   Since $G$ does not admit a clique cutset, it follows that $u$ is 
  non-adjacent to at least one of $x,y$.   If $u$  is adjacent to
  exactly one vertex among $x$ and $y$, then the claim holds.  So we
  may assume that $u$ is adjacent to neither $x$ nor~$y$. 
  
  Since $G$ is \{diamond, $K_4$\}-free, at most one
  vertex of $G$ is adjacent to both $x$ and $y$. Let $a$ be such a 
vertex, if it exists. Let $K=\{x,y,a\}$ if $a$ exists, and let
$K=\{x,y\}$ otherwise. 

Since $K$ is not a clique cutset in $G$, we deduce that  $u$ has a neighbor in 
every component of  $C \setminus K$.
Suppose first that  there is a component $C'$ of $C \setminus K$
containing  a neighbor of one of $x,y$. 
Let $P$ be a path with interior in $C'$, one of whose ends is $u$,
and the other one is in $\{x,y\}$, and subject to that as short as possible.
Then only one of $x,y$ has a neighbor in $V(P) \setminus \{x,y\}$,
and~(\ref{path}) holds. So we may assume that no such component $C'$ exists,
and thus neither of $x,y$ has neighbors in $V(C) \setminus K$.

Let $L=\{a,u\}$ if $a$ exists, and otherwise let $L=\{u\}$.
Then $L$ is a cutset in $G$ separating $C \setminus L$ from $x,y$.
Since $G$ is almost 3-connected, it follows that $L=\{a,u\}$,
and  $C \setminus L$ consists of a unique vertex of degree two, 
so $C$ is degenerate, a contradiction.
\end{proofclaim}

For every $I$-cutset $\{u,x,y\}$, some nondegenerate component $C_1$ of $G \sm \{u,x,y\}$ 
has no vertex with degree two in $G$; choose an $I$-cutset 
$\{u,x,y\}$ and $C_1$ such that $|V(C_1)|$ is minimum.
We refer to this property as the
\emph{minimality of $C_1$}.  Put $G_1=G[C_1 \cup \{u,x,y\}]$, and $G_2 = G\setminus C_1$.  Assume
without loss of generality that $xy$ is an edge, and let $C_2\ne C_1$ be another nondegenerate component.

From~(\ref{path}) and the symmetry between $x,y$, we may assume without loss of generality that there is a chordless path
$Q$ of $G_2$ from  $u$ to $x$ such that
$y$ has no neighbor
in $V(Q)\setminus \{x\}$, and in particular $u,y$ are non-adjacent. Also, since $u,y$ both have neighbors
in $C_2$, there is a chordless path $R$ of $G_2$ between $u,y$ not containing $x$.
Since $u,y$ both have neighbors in $C_1$, there is a chordless path $P$ of $G_1$ between $u,y$ not containing $x$.
Consequently the union of $P,Q$ and the edge $xy$ is a cycle $S$. Let $D$ be the disc bounded by~$S$.

Suppose that some edge of $G_1$ incident with $x$ is in the interior
of $D$, and some other such edge is in the exterior of $D$. By adding
these two edges to an appropriate path within $G[C_1]$, we obtain a
cycle $S_0$ drawn in the plane, such that the path formed by the union
of $xy$ and $P$ crosses it exactly once; and so one of $y,u$ is in the
interior of the disc bounded by $S_0$, and the other in the
exterior. But this is impossible, because $y,u$ are also joined by the
path $R$, which is included in $G_2$ and thus is disjoint from
$V(S_0)$.  We deduce that we may arrange the drawing so that every
edge of $G_1$ incident with $x$ belongs to the interior of $D$. In
addition we may arrange that the edge $xy$ is incident with the
infinite face.

Subject to this condition (and from now on with the drawing fixed),
let us choose $P$ so that $D$ is minimum.  Since $u,x,y\in V(S)$,
every component of $G\setminus V(S)$ has vertex set either a subset of~$C_1$ or disjoint from~$C_1$.  Suppose that some vertex $c$ of $C_1$
is drawn in the interior of $D$, and let $K$ be the component of
$G\setminus V(S)$ containing it. From the choice of $P$, it follows
that there do not exist two non-consecutive vertices of $P$ both with
neighbors in $K$ and, since $|N(K)|\ge 3$ (because $G$ is almost
3-connected and all vertices in $C_1$ are of degree at least three),
and $N(K)\subseteq V(P)\cup \{x\}$, we deduce that $|N(K)| = 3$, and
$N(K) = \{x,a,b\}$ say, where $a,b$ are consecutive vertices of $P$.
From the minimality of $C_1$, $\{x,a,b\}$ is not an I-cutset, and thus
$K$ is degenerate.  Hence, since $K$ has no vertex of degree $2$,
$|V(K)| = 1$, i.e.~$V(K)=\{c\}$.  Therefore $c$ has degree three, with
neighbors $x,a,b$. But then $c$ has three neighbors in $S$, and so $G$
contains a wheel, a contradiction.

Thus no vertex in $C_1$ is drawn in the interior of $D$.  So, since
all edges of $G_1$ incident with $x$ belong to the interior of $D$,
every neighbor of $x$ in $C_1$ belongs to $P$.  Since $G$ is almost
3-connected and all vertices of $C_1$ are of degree at least 3, $x$
has at least one neighbor in~$C_1$.  Since $P\cup R$ is a chordless
cycle, it follows that $x$ has at most two neighbors in $P$ (counting
$y$), and so only one neighbor in $C_1$.  Let $x_1$ be the unique
neighbor of $x$ in~$C_1$.

Since $|V(C_1)|\ge 2$, there is a vertex $x_2$ different from $x_1$ in $C_1$, and since $G$ is almost 3-connected,
there are two paths of $G$, from $x_2$ to $u,y$ respectively, vertex-disjoint except for $x_2$, and not containing $x_1$.
Consequently both these paths are paths of $G_1$, and so there is a path of $G_1$ between $u,y$, containing neither
of $x,x_1$. We may therefore choose a chordless path $P'$  of $G_1$ between $u,y$, containing neither
of $x,x_1$. It follows that the union of $P',Q$ and the edge $xy$ is a chordless cycle $S'$ say, bounding a disc $D'$
say; choose $P'$ such that $D'$ is minimal. 
Since $x_1$ is in $P$ and $xy$ is incident with the infinite face, it follows that $x_1$ is in the interior of $D'$.

Let $Z$ be the set of vertices in $C_1\setminus \{x_1\}$ that are drawn in the
interior of $D'$. We claim that every vertex in $Z$ has degree three, and is adjacent to $x_1$ and to two consecutive
vertices of $P'$. For let $c\in Z$, and let $K$ be the component of $G\setminus (V(S')\cup \{x_1\})$ that contains
$c$. From the choice of $P'$, no two non-consecutive vertices of $P'$ have neighbors in $K$, and so as before,
$N(K) = \{a,b,x_1\}$, where $a,b$ are consecutive vertices of $P'$, and $|V(K)| = 1$. It follows that every vertex in $Z$
has degree three and is adjacent to $x_1$ and to two consecutive vertices of $P'$. 

Let $x_1$ have $t$ neighbors in $P'$. Thus $x_1$ has at least $t+1$
neighbors in the chordless cycle $S'$, and consequently $t\le 1$ since
$G$ does not contain a wheel.  The degree of $x_1$ equals $|Z|+t+1$,
and since $x_1$ has degree at least three and $t\le 1$, we deduce that
$Z\ne \emptyset$, and either $t = 1$, or $t = 0$ and $|Z|>1$.  
Choose $z\in Z$, and let $z$ be adjacent to $a,b,x_1$, where $u,a,b,y$ are in
order in $P'$. 

We claim that $x_1$ is  adjacent to neither $u$ nor $y$. 
For suppose $x_1$ is adjacent to $u$ or $y$.
Since $x_1$ is the unique neighbor of $x$ in $C_1$, $\{u,x_1,y\}$ is a cutset of $G$  separating $C_1 \sm \{x_1\}$ from the rest of the graph. 
So, since it is not an $I$-cutset and since all vertices in $C_1$ have degree at least $3$, $|C_1 \sm \{x_1\}|=1$ and thus $C_1 \sm \{x_1\}=\{z\}$. 
Since $z$ has degree at least $3$, $z$ is adjacent to $u$, $y$ and $x_1$ and, since $z \notin P'$, $P'=uy$. 
Hence $G[\{u,y,z,x_1\}]$ is a diamond, a contradiction to (\ref{diamond}) or else $x_1$ is adjacent to both $u$ and $y$ and $G[\{u,y,z,x_1\}]$ is a wheel, a contradiction. 
So $x_1$ is  adjacent to neither $u$ nor $y$. 

If $x_1$ has a neighbor in $P' \sm \{u,y\}$ (a unique neighbor
because $t\leq 1$) between $u$ and $a$,
say $v$, then $z$ has three neighbors in the chordless cycle formed
by the union of $x_1v$, the subpath of $P'$ between $v$ and $y$, and
the edges $yx$ and $xx_1$. On the other hand, if $x_1$ has a neighbor in $P'$
between $b$ and $y$, say $v$, then $x_1$ has three neighbors in the
chordless cycle formed by the union of $x_1v$, the subpath of $P'$
between $v$ and $u$, the path $Q$ and the edge $xx_1$. Thus $x_1$ has no neighbor in
$P'$, and so $t=0$ and $|Z|\ge 2$.  Let $z'\in Z\setminus \{z\}$,
adjacent to $x_1,a',b'$ say, where $a',b'$ are consecutive vertices of
$P'$, and $u,a',b',y$ are in order on $P'$. From planarity,
$\{a,b\}\ne \{a',b'\}$, and so we may assume that $u,a,a',y$ are in
order on $P'$.  But then $z'$ has three neighbors in the chordless
cycle formed by the path $y\d x\d x_1\d z\d b$ and the subpath of $P'$
between $b$ and $y$, a contradiction.
\end{proof}

\section{A characterization of basic graphs}
\label{sec:equiv}

We need the following. 

\begin{theorem}[Sedla{\v c}ek~\cite{sedlacek:L}]
  \label{th:sed}
  If $H$ is a graph of maximum degree at most three, then $L(H)$ is planar
  if and only if $H$ is planar.
\end{theorem}

  We now prove the following implications between the three statements
  of Theorem~\ref{th:equiv}. 
 
 \medskip{}

  \noindent{\bf($\mathbf{1 \Rightarrow 2}$)}. Suppose that $G$ is a basic graph.  From
    the definition, $G$ is a line graph of a planar graph $R$ of
    maximum degree at most~3. Moreover, it is easy to see that $G$ is 3-connected.
 By Theorem~\ref{th:sed}, $G$ is planar.
    It remains to check that is wheel-free.  If $R=K_{2, 3}$, then
    $G$ is obviously wheel-free.  Otherwise, $R$ is obtained from a
    3-connected cubic planar graph by subdividing every edge exactly
    once.  Suppose for a contradiction that $(u, C)$ is a wheel in
    $G$.  Since $G= L(R)$, $u$ is an edge of $R$, and we set $u=xy$
    where $x$ has degree $3$ and $y$ has degree~$2$.  Let $x'$ be the
    other neighbor of $y$ (so, $x'$ has degree 3 in $R$). In $R$,
    there are two edges $e$ and $f$ different from $xy$ and incident
    to $x$.  And there are two edges $e'$ and $f'$ different from
    $x'y$ and incident to $x'$.  Since $u$ (seen as a vertex of $G$)
    has degree~3, the cycle $C$ of $G$ must go through $e$, $f$ and
    $yx'$ (also seen as vertices of $G$).  But to go in and out from
    the vertex $yx'$ of $G$, the only way is through $e'$ and $f'$ that are
    adjacent.  It follows that $C$ has a chord, a contradiction.

\medskip{}
 
\noindent{\bf($\mathbf{2 \Rightarrow 3}$)}. Suppose that $G$ is a
3-connected wheel-free planar line graph, say $G= L(R)$.  Since $G$ is
a line graph, it is claw-free.  Since $G$ is wheel-free, it is
$K_4$-free.

    Suppose for a contradiction that $G$ contains a diamond.  It
    follows that $R$ contains a paw (see Figure~\ref{f:cdbp}), say a
    triangle $xyz$ and vertex $t$ adjacent to $x$ and to none of $y$
    or $z$.  Since $G$ is 3-connected, the removal of the edge $xt$ in
    $R$ keeps $R$ connected.  It follows that in $R$, there is a path
    $P$ from $t$ to $y$ or $z$, that does not use the edge $tx$.  The
    edges of $P$, together with the edges $tx$, $xy$, $yz$ and $zx$
    form a wheel in $G$, a contradiction.

    Suppose finally that $G$ contains a butterfly. The vertex of
    degree~4 in the butterfly is an edge $xy$ in $R$, and both $x$ and
    $y$ have degree at least~3 (because of the butterfly), and in fact
    exactly~3 (because $G$ contains no $K_4$). Since $G$ is
    3-connected, the removal of the vertex $xy$ of $G$ makes a 
    2-connected graph.  It follows that in $R$, there exists a cycle through
    $x$ and $y$ that does not go through the edge $xy$.  Hence, the
    edges  of this cycle form the rim of a wheel in $G$ (the center is
    the vertex $xy$ of $G$).  This is a contradiction. 

\medskip{}

  \noindent{\bf($\mathbf{3 \Rightarrow 1}$)}. By Theorem~\ref{harary}, $G$ is the line
    graph of a triangle-free graph $R$.  Since $G$ is $K_4$-free,
    every vertex of $R$ has degree at most~$3$.  In particular, since
    $G$ is planar, by Theorem~\ref{th:sed}, $R$ must be planar.  Also, if $R$ has a
    cutvertex $x$, at least one pair of edges incident to $x$ form a
    cutset (of vertices) in $G$, because $G$ has at least four
    vertices since it is 3-connected.  This is a contradiction to the
    3-connectivity of $G$.  It follows that $R$ is 2-connected.

   If two adjacent vertices of $R$ have degree~3, then $G$ contains a diamond or a butterfly, a contradiction.
    Hence, $R$ is edge-wise partitioned into its branches, where a
    \emph{branch} in a graph is a path of length at least~2, whose
    ends have degree at least~3 and whose internal vertices have
    degree~2.
   In fact, every branch of $R$ has
    length exactly~2 because a branch of length at least~3 would yield
    a vertex of degree~2 in $G$, a contradiction to its
    3-connectivity.   

    Suppose that there is a pair of vertices $x, y$ of degree 3 in $R$
    such that at least two distinct branches $P, Q$ have ends $x$ and
    $y$.  We denote by $e$ (resp.\ $f$) the edge incident to $x$
    (resp. $y$) that does not belong to $P$ or $Q$.  Now, $G\setminus
    \{e, f\}$ is disconnected (contradicting $G$ being 3-connected),
    unless $e$ and $f$ are the only edges of $R$ that do not belong to
    $P$ and $Q$.  But in this case, $R = K_{2, 3}$.  So, from here on,
    we may assume that for all pairs of vertices $x, y$ from $R$, there
    is at most one branch of $R$ with ends $x$ and $y$.

    It follows that by suppressing all vertices of degree~2 of $R$, a
    cubic graph $R'$ is obtained (\emph{suppressing} a vertex of
    degree~2 means contracting one the edges incident to it) .
    Suppose that $R'$ is not 3-connected.  This means that
    $R'\setminus \{x, y\}$ is disconnected where $x$ and $y$ are
    vertices of $R'$.  Since $R'$ is cubic, for at least one component
    $C_x$ of $R'\setminus \{x, y\}$, $x$ has a unique neighbor $x'$ in
    $C_x$.  Also, $y$ has a unique neighbor $y'$ in some component
    $C_y$.  Now, $xx'$ and $yy'$ are two edges of $R'$ whose removal
    disconnects $R'$.  These two edges are subdivided in $R$, but they
    still yield two edges whose removal disconnects $R$.  This yields
    two vertices in $G$ whose removal disconnects $G$, a contradiction
    to $G$ being 3-connected.  We proved that $R$ is obtained from a
    3-connected cubic graph (namely $R'$) by subdividing once every
    edge.

\section{Coloring wheel-free planar graphs}
\label{coloration}

A \textit{coloring} of $G$ is a function $\pi: \, V(G) \rightarrow
\mathcal C$ such that no two adjacent vertices receive the same color
$c \in \mathcal C$. If $\mathcal C=\{1, 2, \dots, k\}$, we say that
$\pi$ is a \textit{k-coloring} of $G$.  An \textit{edge-coloring} of
$G$ is a function $\pi: \, E(G) \rightarrow \mathcal C$ such that no
two adjacent edges receive the same color $c \in \mathcal C$. If
$\mathcal C=\{1, 2, \dots, k\}$, we say that $\pi$ is a
\textit{k-edge-coloring} of $G$.
Observe that an edge-coloring of a graph $H$ is also a 
coloring of $L(H)$.  

A graph $R$ is \emph{chordless} if every cycle in $R$ is chordless.  A
way to obtain a chordless graph is to take any graph and to subdivide
all edges.  It follows that basic graphs are in fact line graphs of
chordless graphs.  This is the property of basic graphs that we rely on
in this section. 

It is proved in~\cite{mft:chordless} that for all $\Delta\geq 3$ and
all chordless graphs $G$ of maximum degree $\Delta$, $G$ is
$\Delta$-edge-colorable (for $\Delta=3$, a simpler proof is given in
\cite{nicolas:isk4}).  Unfortunately, this result is not enough for
our purpose and we reprove it for $\Delta=3$ in a slightly more
general form.  A graph is \emph{almost chordless} if at most one of
its edges is the chord of a cycle.

\begin{theorem}
  \label{almostChordless}
  If $G$ is an almost chordless graph with maximum degree three, then $G$
  is 3-edge-colorable.
\end{theorem}

\begin{proof}
  Let $G$ be a counter-example with the minimum number of edges.  Let $X
  \subseteq V(G)$ be the set of vertices of degree three and $Y=V(G)
  \sm X$ the set of vertices of degree at most two.

\begin{claim}  \label{Y}
\, $Y$ is a stable set.
\end{claim}

\begin{proofclaim}
  For suppose that there exists an edge $uv$ such that $u$ and $v$
  belong to $Y$.  From the minimality of $G$ 
there exists a 3-edge-coloring of $G
  \sm uv$.  Since $u,v\in Y$,
it is easy to extend the 3-edge-coloring of $G \sm uv$ to a
  3-edge-coloring of $G$, a contradiction.
\end{proofclaim}

\begin{claim}
  \label{c:2con}
  \, $G$ is 2-connected.
\end{claim}

\begin{proofclaim}
  Otherwise $G$ has a cut-vertex $v$, so $V(G)\sm \{v\}$ partitions into two
  nonempty sets of vertices $C_1$ and $C_2$ with no edges between them.  A
  3-edge-coloring of $G$ can be obtained easily from 3-edge-colorings
  of $G[C_1\cup \{v\}]$ and $G[C_2 \cup \{v\}]$, a contradiction.
\end{proofclaim}

\begin{claim}\label{edgeCutset}
\, If $e,f$ are disjoint edges of $G$, then $G\setminus \{e,f\}$ is connected. 
\end{claim}

\begin{proofclaim}
  Suppose there exists two disjoint edges $u_1u_2$ and $v_1v_2$ such
that $G\setminus \{u_1u_2,v_1v_2\}$ is not connected; then $G\sm
\{u_1u_2, v_1v_2\}$  partitions into two nonempty sets of vertices
  $C_1$ and $C_2$ with no edges between them.  By~(\ref{c:2con}) we
  may assume that $\{u_1,v_1\} \subseteq C_1$ and $\{u_2,v_2\}
  \subseteq C_2$.  For $i=1,2$, let $G_i$ be the graph obtained from
  $G[C_i]$ by adding a vertex $m_i$ adjacent to both $u_i$ and $v_i$.
  If $G_1$ contains a cycle $C$ with a chord $ab$, then $ab$ is a
  chord of a cycle of $G$ (this is clear when $C$ does not contain
  $m_1$, and when $C$ contains $m_1$, the cycle is obtained by
  replacing $m_1$ by a $u_2v_2$-path included in $C_2$
  that exists by~(\ref{c:2con})).  It follows that $G_1$ and
  symmetrically $G_2$ are almost chordless.  Moreover they both
  clearly have maximum degree at most three and, by~(\ref{Y}), both $C_1$ and
  $C_2$ contain vertices of degree  three, so $G_1$ and $G_2$ have fewer
  edges than $G$.  Therefore $G_1$ and $G_2$ admit  3-edge-colorings.
  
  Let $\pi_1$ and $\pi_2$ be 3-edge-colorings of respectively $G_1$
  and $G_2$.  We may assume without loss of generality that
  $\pi_1(u_1m_1)=\pi_2(u_2m_2)=1$ and $\pi_1(v_1m_1)=\pi_2(v_2m_2)=2$.
  Now, the following coloring $\pi$ is a 3-edge-coloring of $G$:
  $\pi(u_1v_1)=1$, $\pi(u_2v_2)=2$, $\pi(e)=\pi_1(e)$ if $e \in
  E(G_1)$ and $\pi(e)=\pi_2(e)$ if $e \in E(G_2)$, a contradiction.
\end{proofclaim}

\begin{claim} \label{X}
\, $G[X]$ has at most one edge, and if it has one, it is a chord of a
cycle of $G$.  
\end{claim}

\begin{proofclaim}
  Suppose that $xy$ is an edge of $G[X]$ such that $G \sm xy$ is not
  2-connected.  Then, there exists a vertex $w$ such that $G \sm
  \{xy,w\}$ is disconnected.  Let $C_x$ and $C_y$ be  the two 
  components of $G \sm \{xy,w\}$, where $x \in C_x$ and $y \in
  C_y$.  Since $w$ is of degree at most three, $w$ has a unique neighbor
  $w'$ in one of $C_x,C_y$, say in $C_x$.  
If $w'=x$, then $x$ is a cut-vertex of $G$ (because $|C_x|>1$ since $x$ has degree three), a contradiction to~(\ref{c:2con}).
So $w' \neq x$ and hence $xy, ww'$ are disjoint, a contradiction to (\ref{edgeCutset}).

  Therefore, for every edge $xy$ of $G[X]$, $G \sm xy$ is 2-connected.
  So, if such an edge exists, by Menger's theorem there exists a
  cycle $C$ going through both $x$ and $y$ in $G \sm xy$, and thus $xy$
  is a chord of $C$.  Since $G$ is almost chordless, there is at most
  one such edge.
\end{proofclaim}

If $G$ is chordless, then by~(\ref{Y}) and~(\ref{X}), $(X, Y)$ forms a
bipartition of $G$, so by a classical theorem of K\H onig, $G$ is
3-edge-colorable, a contradiction.  So let $xy$ be a chord of a cycle
of $G$.  Let $x'$ and $x''$ be the two neighbors of $x$ distinct from
$y$ and let $y'$ and $y''$ be the two neighbors of $y$ distinct from
$x$.  By (\ref{X}), $x'$, $x''$, $y'$ and $y''$ are all of degree 2
and by (\ref{Y}), they induce a stable set.  If
$\{x',x''\}=\{y',y''\}$, then $G$ is the diamond and thus is
3-edge-colorable.  If $|\{x',x''\} \cap \{y',y''\}|=1$, say $x'=y'$
and $x'' \neq y''$, then $xx'',yy''$ are disjoint and their deletion disconnects $G$, a
contradiction to~(\ref{edgeCutset}).  Hence $x'$, $x''$, $y'$ and
$y''$ are pairwise distinct.

Let $x_1'$ (resp.\ $x''_1$, $y'_1$, $y''_1$) be the unique neighbor of
$x'$ (resp.\ $x''$, $y_1'$, $y_1''$) distinct from $x$ (resp.\ $y$).
Let $G'$ be the graph obtained from $G$ by deleting the edge $xy$ and
contracting edges $xx'$, $xx''$, $yy'$ and $yy''$.   
We note $x$ the vertex resulting from the contraction of $xx'$ and $xx''$, and $y$ the vertex resulting from the contraction of $yy'$ and $yy''$.
Since $G'$
has maximum degree at most three and is bipartite by~(\ref{Y})
and~(\ref{X}), it follows that $G'$ has a 3-edge-coloring $\pi'$ by K\H onig's
theorem.

Assume without loss of generality that $\pi'(xx_1')=1$, $\pi'(xx_1'')=2$,
$\pi'(yy_1')=a$ and $\pi'(yy_1'')=b$ where $\{a,b\} \subseteq
\{1,2,3\}$.  Since $\{a,b\} \cap \{1,2\} \neq \emptyset$, we may
assume without loss of generality that $a=1$, so $b\neq 1$.  Let us now extend this
coloring to a 3-edge-coloring $\pi$ of $G$.  For any edge $e$ of $G$
such that its extremities are not both in
$\{x,y,x',x'',y',y'',x_1',x_1'',y_1',y_1''\}$, set $\pi(e)=\pi'(e)$.
Set $\pi'(x'x_1')=1$, $\pi'(x''x_1'')=2$, $\pi'(y'y_1')=1$ and
$\pi'(y''y_1'')=b$.  Now we can set $\pi(xx')=2$, $\pi(xx'')=1$,
$\pi(yy')=2$,  $\pi(yy'')=1$ and $\pi(xy)=3$.  So $\pi$
is a 3-edge-coloring of $G$.
\end{proof}

Note that in the next proof, we do not use planarity, except when we apply
Theorem~\ref{claw}. 

\bigskip

\noindent{\bf Proof of Theorem~\ref{th:col}\ \ }

We argue by induction on $|V(G)|$.  Suppose first that $G$ admits a
clique cutset $K$.  Let $C_1$ be the vertex set of a component of $G \sm K$
and $C_2=V(G)\sm (K \cup C_1)$.  By induction $G[C_1 \cup K]$ and $G[C_2
\cup K]$ are both 3-colorable and thus $G$ is 3-colorable.  So we may
assume that $G$ has no clique cutset.  If $G$ has a vertex
$u$ of degree two, then we can 3-color $G \sm \{u\}$ by induction and
extend the coloring to a 3-coloring of $G$.  So we may assume that
every vertex of $G$ has degree at least three.

Assume now that $G$ is 3-connected.  By Theorem~\ref{claw}, there
exists a chordless graph $H$ of maximum degree three such that $G=L(H)$.
Hence, by Theorem~\ref{almostChordless}, $H$ is 3-edge-colorable and
thus $G$ is 3-colorable.  So we may assume that the connectivity of
$G$ is two.

Let $\{a,b\} \subseteq V(G)$ be such that $G\sm \{a, b\}$ is
disconnected.  We choose $\{a, b\}$ to minimize the smallest order of a  
component of $G\sm \{a,b\}$, and let $C$ be the vertex set of this component.  If $|C|=1$, then the vertex in $C$
is of degree two in $G$, a contradiction.  So $|C| \ge 2$.
Let $G'_C$ be the graph obtained from $G[C \cup \{a,b\}]$ by adding
the edge $ab$ (that did not exist since $G$ has no clique cutset).
Let us prove that $G'_C$ is 3-connected.  Since $|C| \ge 2$ and $G'_C$ therefore has at least four vertices,
we may assume by contradiction that $G'_C$ admits a 2-cutset $\{x,y\}$.
  Let $C_1, \dots, C_k$ ($k\ge 2$) be the vertex sets of the 
components of $G'_C \sm \{x,y\}$.  Since $ab$ is an edge of $G'_C$,
$a$ and $b$ are included in $G'_C[C_i \cup \{x,y\}]$ for some $i \le
k$, say $i=2$.  Hence $\{x,y\}$ is a cutset of $G$ and $C_1$ is a
component of $G \sm \{x,y\}$ that is a proper subset of $C$,
a contradiction to the minimality of $C$.  So $G'_C$ is 3-connected. (But it might not be wheel-free.)

Let $G_C$ be the graph obtained from $G'_C$ by subdividing $ab$ once, and let
$m$ be the vertex of degree two of $G_C$.  Since $G'_C$ is
3-connected, $G_C$ is almost 3-connected.  Suppose that $G_C$ admits
a wheel $(u, R)$.  Since $G$ is wheel-free, $m$ must be a vertex of
$(u,R)$.  Since $m$ is of degree two, $m$ is in $R$, and so $a\d m\d b$ is a
subpath of $R$.  Since $G$ is 2-connected, there exists a chordless $ab$-path
$P$ in $G\sm C$.  Hence by replacing $a\d m\d b$ by $P$, we obtain a wheel in
$G$, a contradiction.  Therefore $G_C$ is an almost 3-connected  wheel-free
 planar graph.

By Theorem \ref{claw}, there exists a chordless graph
$H$ of maximum degree three such that $L(H)=G_C$.  We are now going to
prove there exist two ways to 3-edge-color $H$, one giving the same
color to $a$ and $b$ (that are edges of $H$), and the other giving
distinct colors to $a$ and $b$.  This implies that there exist two ways
to 3-color $G[C \cup \{a,b\}]$, one giving the same color to $a$ and
$b$ and the other giving distinct colors to $a$ and $b$. Since by
the inductive hypothesis there exists a 3-coloring of $G \sm C$, it follows that this
3-coloring can be extended to a 3-coloring of $G$.

We first prove that there exists a 3-edge-coloring $\pi$ of $H$ such
that $\pi(a)\ne\pi(b)$.  Observe that both ends of $m$ are of
degree two in $H$.  Hence, $H/m$ is also a chordless graph with maximum
degree at most three.  Therefore there exists a 3-edge-coloring $\pi$ of
$H /m$ and clearly $\pi$ statisfies $\pi(a) \neq \pi(b)$.  It is easy
to extend $\pi$ to a 3-edge-coloring of $H$ by giving a color distinct
from $\pi(a)$ and $\pi(b)$ to $m$.

Let us now prove that there is a 3-edge-coloring of $H$ such that
$\pi(a)=\pi(b)$.  Let $m=m_am_b$, $a=m_aa_1$ and $b=m_bb_1$.  
We claim that $a_1b_1$ is not an edge of $H$. 
For if $a_1b_1$ is an edge of $H$, then there exists a vertex $x$ in $G_C$ adjacent to both $a$ and $b$. 
Since $G_C$ is almost 3-connected and $m$ is the only vertex of degree $2$ in $G_C$, $G_C \sm \{x,m\}$ is connected, and thus there exists a path $P$ between $a$ and $b$ avoiding $x$ and $m$. 
Since $a$ and $b$ are not adjacent, $P$ is of length at least $2$. 
Naming $u$ the vertex of $P$ adjacent to $a$, $u$ is adjacent to $x$, otherwise $G_C[\{a,x,u,m\}]$ is a claw of $G_C$, contradicting the fact that $G_C$ is a line graph. 
Hence $x$ has at least three neighbors in the chordless cycle formed by the path $P$ and the edges $am$ and $bm$, a contradiction to the fact that $G_C$ is wheel-free.  
So $a_1b_1$ is not an edge of $H$.

Let $H'$
be the graph obtained from $H$ by deleting the vertices $m_a$ and $m_b$
and adding the edge $a_1b_1$.  If an edge $xy$ distinct from $a_1b_1$
is the chord of a cycle $Q$, then since it is not a chord in $H$, $Q$
must contain $a_1b_1$.  Then by replacing $a_1b_1$ by $a_1\d m_a\d m_b\d b_1$,
we deduce that $xy$ is also the chord of a cycle in $H$, a contradiction.
Hence $H'$ is almost chordless and thus, by
Theorem~\ref{almostChordless}, $H'$ admits a 3-edge-coloring $\pi'$.
Assume that $\pi'(a_1b_1)=1$.  Then setting
$\pi(a_1m_a)=\pi(b_1m_b)=1$ and $\pi(m_am_b)=2$, we obtain a 3-edge-coloring 
of $H$ satisfying $\pi(a_1m_a)=\pi(b_1m_b)$. This completes the proof of 
Theorem~\ref{th:col}. \hfill$\Box$

\end{document}